\documentclass[11pt]{amsart}
\usepackage{graphicx}
\usepackage{amscd}
\usepackage{amsmath}
\usepackage{amsxtra}
\usepackage{amsfonts}
\usepackage{amssymb}

\newtheorem{theorem}{Theorem}[section]
\newtheorem{corollary}[theorem]{Corollary}
\newtheorem{lemma}[theorem]{Lemma}
\newtheorem{proposition}[theorem]{Proposition}

\theoremstyle{definition}
\newtheorem{definition}[theorem]{Definition}
\newtheorem{remark}[theorem]{Remark}

\newtheorem{example}[theorem]{Example}
\theoremstyle{remark}

\renewcommand{\theclaim}{\textup{\theclaim}}

\newtheorem*{acknowledgements}{Acknowledgements}

\numberwithin{equation}{section}

\def\openone

{\mathchoice

{\hbox{\upshape \small1\kern-3.3pt\normalsize1}}

{\hbox{\upshape \small1\kern-3.3pt\normalsize1}}

{\hbox{\upshape \tiny1\kern-2.3pt\SMALL1}}

{\hbox{\upshape \Tiny1\kern-2pt\tiny1}}}

\makeatletter

\newbox\ipbox

\newcommand{\ip}[2]{\left\langle #1\, , \,#2\right\rangle}
\newcommand{\diracb}[1]{\left\langle #1\mathrel{\mathchoice

{\setbox\ipbox=\hbox{$\displaystyle \left\langle\mathstrut
#1\right.$}

\vrule height\ht\ipbox width0.25pt depth\dp\ipbox}

{\setbox\ipbox=\hbox{$\textstyle \left\langle\mathstrut
#1\right.$}

\vrule height\ht\ipbox width0.25pt depth\dp\ipbox}

{\setbox\ipbox=\hbox{$\scriptstyle \left\langle\mathstrut
#1\right.$}

\vrule height\ht\ipbox width0.25pt depth\dp\ipbox}

{\setbox\ipbox=\hbox{$\scriptscriptstyle \left\langle\mathstrut
#1\right.$}

\vrule height\ht\ipbox width0.25pt depth\dp\ipbox}

}\right. }

\newcommand{\dirack}[1]{\left. \mathrel{\mathchoice

{\setbox\ipbox=\hbox{$\displaystyle \left.\mathstrut
#1\right\rangle$}

\vrule height\ht\ipbox width0.25pt depth\dp\ipbox}

{\setbox\ipbox=\hbox{$\textstyle \left.\mathstrut
#1\right\rangle$}

\vrule height\ht\ipbox width0.25pt depth\dp\ipbox}

{\setbox\ipbox=\hbox{$\scriptstyle \left.\mathstrut
#1\right\rangle$}

\vrule height\ht\ipbox width0.25pt depth\dp\ipbox}

{\setbox\ipbox=\hbox{$\scriptscriptstyle \left.\mathstrut
#1\right\rangle$}

\vrule height\ht\ipbox width0.25pt depth\dp\ipbox}

} #1\right\rangle}

\newcommand{\beq}{\begin{equation}}

\newcommand{\eeq}{\end{equation}}

\newcommand{\cj}[1]{\overline{#1}}

\newcommand{\bz}{\mathbb{Z}}

\newcommand{\bc}{\mathbb{C}}

\newcommand{\bn}{\mathbb{N}}

\def\blfootnote{\xdef\@thefnmark{}\@footnotetext}

\newcommand{\suport}{\textup{supp}}

\renewcommand{\mod}{\operatorname{mod}}

\input xy
\xyoption{all}
\usepackage{amssymb}

\def\-{^{-1}}

\def\ty{\emptyset}

\begin{document}

\title[On generalized Walsh bases]{On generalized Walsh bases}
\author{Dorin Ervin Dutkay}

\address{[Dorin Ervin Dutkay] University of Central Florida\\
	Department of Mathematics\\
	4000 Central Florida Blvd.\\
	P.O. Box 161364\\
	Orlando, FL 32816-1364\\
U.S.A.\\} \email{Dorin.Dutkay@ucf.edu}

\author{Gabriel Picioroaga}
\address{
[Gabriel Picioroaga] University of South Dakota\\
Department of Mathematical Sciences\\
414 E. Clark Street\\
Vermillion, SD, 57069\\
U.S.A\\}\email{Gabriel.Picioroaga@usd.edu}

\author{Sergei Silvestrov}
\address{[Sergei Silvestrov] Division of Applied Mathematics\\
The School of Education, Culture and Communication (UKK)\\
M\" alardalen University\\
Box 883, 721 23 V\" aster\aa s, Sweden} \email{sergei.silvestrov@mdh.se}

\thanks{}
\subjclass[2000]{}

\begin{abstract} This paper continues the study of orthonormal bases (ONB) of $L^2[0,1]$ introduced in \cite{DPS14} by means of Cuntz algebra $\mathcal{O}_N$ representations on $L^2[0,1]$. For $N=2$, one obtains the classic Walsh system. We show that the ONB property holds precisely because the $\mathcal{O}_N$ representations are irreducible. We prove an uncertainty principle related to these bases. As an application to discrete signal processing we find a fast generalized transform and compare this generalized transform with the classic one with respect to compression and sparse signal recovery. 

\end{abstract}
\maketitle \tableofcontents
\section{Introduction}

The Walsh  functions form an orthonormal basis (ONB), for the Hilbert space $L^2[0,1]$,  that can be interpreted roughly as the discrete analog of classic sines and cosines. For some applications, the Walsh functions have several advantages: for example they take only the values $\pm 1$ on sub-intervals defined by dyadic fractions, thus making the computation of coefficients much easier. The Walsh functions  are connected to probability, e.g., the Walsh expansion can be seen as conditional expectation, and the partial sums form a Doob martingale.
The Walsh functions have found a wide range of applications: for example in modern communications systems (through the so-called Hadamard matrices, to recover information in the presence of noise and interference), and signal processing (reconstruction of signals by means of dyadic sampling theorems), see e.g. \cite{ARa}, \cite{Corr}, \cite{Hart}, \cite{Yuen} to mention a few. 

There are certain features that make this ONB more desirable to work with than for example the Fourier system: the Walsh series associated to $f\in L^1[0,1]$ converge pointwise a.e. to $f$. This is also true for $f$ with bounded variation at a continuity point of $f$, e.g. see \cite{Walsh}, \cite{Paley}, \cite{Nagy}.

Various generalizations have been given, based on changing the $L^2$ space \cite{Vil}, or changing the Rademacher functions \cite{Chr}. For example, for the dyadic group $G$ the Walsh functions can be viewed as characters on $G$, or more generally starting with \cite{Vil}, as characters of a zero-dimensional, separable group. The  generalized Walsh system based on $N$-adic numbers and exponentials functions in \cite{Chr} has been used to construct algorithms for polynomial lattices (a particular kind of digital net which in turn can be used in sampling methods for multivariate integration), see \cite{Dik1}, \cite{Dik2} and references therein.
\par In \cite{DPS14} a criteria is given to obtain ONBs from Cuntz algebra representations. These bases were obtained through a general principle which incorporates wavelets and an assortment of various ONBs, from the classic Fourier and Walsh bases, to bases on fractals (Cantor sets). This principle is based on the Cuntz relations and roughly asserts that a suitable representation of these relations gives rise to a ONB. By tinkering with the isometries satisfying the relations one obtains the above mentioned variety of ONBs. One of the examples recovers the classic Walsh ONB, and another generalizes it (the ONB from \cite{Chr} is also recovered by starting with a $N\times N$ Hadamard matrix in the construction below). It is this generalized ONB that we continue to study in this paper. We mention that properties such as convergence, continuity, periodicity etc. are discussed in \cite{DuPi} and \cite{NHa}. 
\par  In Section \ref{s2} we show that the ONB property of the generalized Walsh system is equivalent with the irreducibility of the  Cuntz algebra representation from which it is built. In Section \ref{s3} we analyze the restriction of the generalized Walsh bases to finite-dimensional spaces. We show that the signal's coefficients with respect to these  bases can be easily read off of a tensor matrix. We also provide a change of generalized Walsh bases formula. From these finite-dimensional considerations, in Section \ref{fast} we provide an algorithm for a fast generalized Walsh transform, generalizing ideas from \cite{LeKa86}. The fast transform was implemented in Maple and used  in all our examples from section \ref{s6} to compare the performances of Walsh and generalized Walsh transforms from a statistical point of view. To this end we use a simple compression scheme based on the variance criterion. In Section \ref{incert}, inspired by a discrete version of the uncertainty principle \cite{DoSt} we  show that the Walsh and generalized Walsh transforms satisfy such a principle albeit a  discrepancy which occurs when the matrix giving rise to the generalized Walsh basis is not Hadamard. We further develop a concept of uncertainty with respect to unitary matrices and prove a few properties. 
\par We recall next the construction of the generalized Walsh basis and the Cuntz algebra representation it arises from.

Let $A=(a_{ij})_{i,j=0,\dots,N-1}$ be an $N\times N$ unitary matrix with constant first row, i.e., $a_{0j}=\frac{1}{\sqrt N}$ for all $j=0,\dots,N-1$.

Define the functions (almost everywhere with respect to the Lebesgue measure):
\begin{equation}
m_i(x)=\sqrt N\sum_{j=0}^{N-1}a_{ij}\chi_{[j/N,(j+1)/N)}(x),\quad(x\in [0,1],i=0,\dots,N-1),
\label{eq1.1}
\end{equation}
where $\chi_A$ is the characteristic function of the set $A$.

Note that $m_0\equiv 1$.

On $L^2[0,1]$ define the operators
\begin{equation}
S_if(x)=m_i(x)f(Nx\mod 1)\quad (x\in [0,1],f\in L^2[0,1], i=0,\dots,N-1).
\label{eq1.2}
\end{equation}
Note that $S_01=1$.

\begin{theorem}\label{thcuntz}\cite{DPS14}
The operators $(S_i)_{i=0,1,\dots,N-1}$ form a representation of the Cuntz algebra $\mathcal O_N$ on  $L^2[0,1]$, i.e.,
\begin{equation}
S_i^*S_j=\delta_{i,j}I_{L^2[0,1]},\quad(i,j\in\{0,1,\dots,N-1\}), \quad \sum_{i=0}^{N-1}S_iS_i^*=I_{L^2[0,1]}.
\label{eqcuntz}
\end{equation}

\end{theorem}

\begin{theorem}\label{thdps}\cite{DPS14}
The family
\begin{equation}
\mathcal W^A:=\left\{ S_{w_1}\dots S_{w_p}1 : p\in \bn, w_1,\dots, w_n\in\{0,\dots,N-1\}\right\}
\label{eq1.3}
\end{equation}
is an orthonormal basis for $L^2[0,1]$ (discarding of course the repetitions generated by the fact that $S_01=1$). We call it the generalized Walsh basis that corresponds to the matrix $A$.
\end{theorem}

\section{The representation of the Cuntz algebra $\mathcal O_N$}\label{s2}
In this section we study the representation of the Cuntz algebra defined by the isometries in \eqref{eq1.2}. We show that this is a permutative representation as the ones studied in \cite{MR1465320,DHJ15,MR1980913,MR2531317,MR2560855}.
We show that the ONB property is present precisely because the associated representation of the Cuntz algebra $\mathcal{O}_N$ on $L^2[0,1]$ is irreducible. We first find an irreducible representation, then show it is equivalent with the one in (\ref{eq1.2}).

Let $\Omega$ be the set of finite words with digits in $\{0,1,\dots,N-1\}$ and that do not end in 0, including the empty word $\ty$. Define the canonical vectors in $l^2(\Omega)$, 
$$e_\omega(\gamma)=\left\{\begin{array}{cc}
1,&\mbox{ if }\gamma=\omega\\
0,&\mbox{ if }\gamma\neq\omega.\end{array}\right., \quad(\omega,\gamma\in\Omega).$$
Define the linear operators $s_i$, $i=0,1,\dots,N-1$ on $l^2(\Omega)$ by
\begin{equation}
s_0e_\omega=\left\{\begin{array}{cc}
e_\ty,&\mbox{ if }\omega=\ty\\
e_{0\omega},&\mbox{ otherwise.}\end{array}\right.,\quad s_ie_\omega=e_{i\omega}, \quad(\omega\in \Omega,i\neq 0),
\label{eq3.1}
\end{equation}
($i\omega$ represents here concatenation of the digit $i$ to the word $\omega$). 

\begin{theorem}\label{th3.1}
The operators $(s_i)_{i=0}^{N-1}$ form an irreducible representation of the Cuntz algebra $\mathcal O_N$ on $l^2(\Omega)$.

If $A$ is a unitary matrix with constant first row $1/\sqrt{N}$ and $(S_i)_{i=0}^{N-1}$ is the associated representation of the Cuntz algebra as in \eqref{eq1.2}, then the linear operator $W:L^2[0,1]\rightarrow l^2(\Omega)$, defined by
\begin{equation}
W(S_\omega1)=e_\omega,\quad(\omega\in\Omega),
\label{eq3.1.1}
\end{equation}
is unitary and intertwines the two representations.
\end{theorem}

\begin{proof}
We prove that the operators $(s_i)_{i=0,\dots,N-1}$ satisfy the Cuntz relations. It is easy to see that the operators are norm preserving, hence isometries, and that their ranges are orthogonal. Thus $s_i^*s_j=\delta_{i,j}I$. Also we can see that
$s_0^*e_{\ty}=s_0^*s_0e_{\ty}=e_{\ty}$, if $\omega=\omega_1\omega_2\dots\omega_n\neq\ty$ then $$s_0^*e_\omega=s_0^*s_{\omega_1}e_{\omega_2\dots\omega_n}=\delta_{0,\omega_1}e_{\omega_2\dots\omega_n},$$
and, for $i\neq 0$, $s_i^*e_\ty=0$, and if $\omega=\omega_1\omega_2\dots\omega_n\neq\ty$ then
$$s_i^*e_{\omega}=\delta_{i,\omega_1}e_{\omega_2\dots\omega_n}.$$
With this, we can easily check that
$$\sum_{i=0}^{N-1}s_is_i^*e_{\omega}=e_{\omega},\quad(\omega\in\Omega).$$

To prove that the representation is irreducible, we use \cite[Theorem 5.1]{BJK00}. The subspace $\mathcal K$ spanned by the vector $e_\ty$ is invariant for the operators $S_i^*$ and cyclic for the representation. Indeed $s_i^*e_\ty=\delta_{0,i}e_\ty$, $i=0,1,\dots,N-1$. Also $s_\omega\mathcal K$ contains $e_\omega$ so $\mathcal K$ is cyclic for the representation.

Define the operator $V_i^*$ on $\mathcal K$ by $V_i^*\xi=S_i^*\xi$, $\xi\in \mathcal K, i=0,1,\dots,N-1$. By \cite[Theorem 5.1]{BJK00}, the commutant of the representation is in linear one-to-one correspondence with the fixed points of the map on linear operators on $\mathcal K$: $A\mapsto \sum_{i=0}^{N-1}V_iAV_i^*$. Since $\mathcal K$ is one-dimensional, the space of the fixed points of this map is one dimensional, and therefore the commutant is trivial and the representation is irreducible.

Since $\{S_\omega 1 :\omega\in\Omega\}$ is an orthonormal basis, the operator $W$ is unitary and a simple check shows that $WS_i(S_{\omega'}1)=s_i W(S_{\omega'})1$, and   $WS_i^*(S_{\omega}1)=s_i^* W(S_{\omega}1)$ for all $\omega\in\Omega$, $i=0,1,\dots,N-1$, which shows that $W$ intertwines the two representations.
\end{proof}

The construction of the generalized Walsh basis and the previous theorem suggest to study conditions under which vector families of the type  $\{S_{\omega}v_0\}$ become ONB. Theorem \ref{th.3.2} below provides a partial answer. We will need the following lemma (which is more general than Lemma 5 in \cite{PW} whose proof however is essentially the same).
\begin{lemma} Let $S_i:H\to H$, $i=0,..,N-1$ be a representation of the Cuntz algebra $\mathcal{O}_N$ on the separable Hilbert space $H$, and $ v_0\in H$ such that $\|v_0\|=1$ and $S_0(v_0)=v_0$. Then the family $$\{S_{\omega}v_0 : \mbox{ $\omega$ is a finite word over $\{0,1,..., N-1\}$ }\}$$  (with repeats removed) is orthonormal.
\end{lemma}
\begin{proof}
Notice $\ip{S_{\omega} v_0}{S_{\gamma} v_0}=0$ whenever the first digit in $\omega$ differs from the first digit in $\gamma$. The same is true if $\omega_k\neq \gamma_k$ for some $k$. The remaining case is when $\omega$ is a prefix in $\gamma$ or vice versa. Assuming the former, the inner product reduces to one of the form
 $\ip{v_0}{S_{\gamma'} v_0}$ or its conjugate. Because $S_0v_0=v_0$, the last inner product is again zero unless $$\gamma'= \underbrace{0\dots0}_{n\geq 0\mbox{ times}},$$ and in this case $S_{\gamma'}v_0=v_0$, but the repetitions are removed.
\end{proof}

\begin{theorem}\label{th.3.2}
Let $\rho:=(S_i:H\to H)_{_{i=0,..,N-1}}$ be a representation of the Cuntz algebra $\mathcal{O}_N$ on the separable Hilbert space $H$, and $ v_0\in H$ such that $\|v_0\|=1$ and $S_0v_0=v_0$. The following statements are equivalent:
\begin{enumerate}
	\item The family $\{S_{\omega}v_0 : \mbox{ $\omega$ word over $\{0,1,..., N-1\}$}\}$ (with repeats removed) is an ONB.
	\item The representation $\rho$ is irreducible.
\end{enumerate}

\begin{proof} The implication (i) $\Rightarrow$ (ii) follows by repeating the argument in the proof of Theorem \ref{th3.1}, i.e., $\rho$ is equivalent to an irreducible representation. For the converse, let $K$ be the closed space spanned by all the vectors $S_\omega v_0$. Clearly this space is invariant under all the isometries $S_i$ and their adjoints $S_i^*$ (due to the Cuntz relations). Therefore this subspace is invariant under the representation $\rho$, so it must be equal to the entire space $H$, because the representation is irreducible.

\end{proof}
\end{theorem}

\section{Computations of generalized Walsh bases and of coefficients}\label{s3}
We can index the generalized Walsh basis by nonnegative integers $\bn_0$ as follows: for $n\in\bn_0$, write the base-$N$ representation
\begin{equation}
n=n_1+Nn_2+\dots+N^pn_{p+1}+\dots, \mbox{ where $n_k=0$ for $k$ large enough.}
\label{eq1.4}
\end{equation}
Then
\begin{equation}
W_n^A=S_{n_1}\dots S_{n_p}1,\mbox{ if $n_k=0$ for all $k>p$.}
\label{eq1.5}
\end{equation}
Note that it does not matter if we pick a larger $p$ in \eqref{eq1.5}, because $S_01=1$.

We compute the generalized Walsh functions more explicitly.
\begin{proposition}\label{pr1.2}
Let $x$ be in $[0,1]$. We write the base $N$ representation of $x$
\begin{equation}
x=\frac{x_1}{N}+\frac{x_2}{N^2}+\dots,\quad (x_1,x_2,\dots\in\{0,\dots,N-1\})
\label{eq1.6}
\end{equation}
(We can ignore the cases when there are multiple representations of the same point, since we are working in $L^2[0,1]$ and the set of such points has Lebesgue measure zero).
For $n\in\bn_0$, consider the base $N$ representation as in \eqref{eq1.4}. Then
\begin{equation}
W_n^A(x)=\sqrt{N^p}a_{n_1x_1}a_{n_2x_2}\dots a_{n_px_p}, \mbox{ if $n_k=0$ for $k>p$}
\label{eq1.7}
\end{equation}
(Since $n_k=0$ for $k$ large, $\sqrt{N}a_{n_kx_k}=1$ for $k$ large, so we can use a larger $p$ in \eqref{eq1.7} if needed).
In other words,
\begin{equation}
W_n^A(x)=\sqrt{N^p}\sum_{j_1,\dots,j_p=0}^{N-1}a_{n_1 j_1}a_{n_2j_2}\dots a_{n_pj_p}\chi_{\left[\frac{j_1}{N}+\frac{j_2}{N^2}+\dots+\frac{j_p}{N^p},\frac{j_1}{N}+\frac{j_2}{N^2}+\dots+\frac{j_p}{N^p}+\frac{1}{N^p}\right]}(x),
\label{eq1.8}
\end{equation}
and here $p$ has the same meaning as in \eqref{eq1.7}.
\end{proposition}

\begin{proof}
By \eqref{eq1.5}, with $n$ as in \eqref{eq1.5}, we have
$$W_n^A=S_{n_1}\dots S_{n_p}1,\mbox{ if $n_k=0$ for $k>p$}.$$
Then $$W_n^A(x)=m_{n_1}(x)m_{n_2}(Nx\mod1)\quad m_{n_p}(N^px\mod1).$$

If $x$ is as in \eqref{eq1.6}, then $x\in[{x_1}/N,(x_1+1)/N)$ so $m_{n_1}(x)=\sqrt{N}a_{n_1x_1}$. Then $Nx\mod1= x_2/N+x_3/N^2+\dots$ so $m_{n_2}(Nx\mod1)=\sqrt{N}a_{n_2x_2}$; by induction, $N^px\mod 1=x_p/N+x_{p+1}/N^2+\dots$ and  $m_{n_p}(N^px\mod1)=\sqrt{N} a_{n_px_p}$.

It follows that
$$W_n^A(x)=\sqrt{N^p}a_{n_1x_1}a_{n_2x_2}\dots a_{n_px_p}.$$

\end{proof}

\begin{definition}\label{deftensor}
Let $A$ be an $N\times N$ matrix and $B$ an $M\times M$ matrix. Then the tensor product of the two matrices $A\otimes B$ is an $NM\times NM$ matrix with entries
$$(A\otimes B)_{i_1+Ni_2,j_1+N j_2}=A_{i_1j_1}B_{i_2j_2},\quad(i_1,j_1=0,\dots,N-1,\quad i_2,j_2=0,\dots,M-1).$$
Thus the matrix $A\otimes B$ has the block form
$$A\otimes B=\begin{pmatrix}
	Ab_{0,0}&Ab_{0,1}&\dots &Ab_{0,M-1}\\
	Ab_{1,0}&Ab_{1,1}&\dots &Ab_{1,M-1}\\
	\vdots&\vdots&\ddots&\vdots\\
	Ab_{M-1,0}&Ab_{M-1,1}&\dots &Ab_{M-1,M-1}\\
\end{pmatrix}.$$
Note that the tensor operation is associative, $(A\otimes B)\otimes C=A\otimes (B\otimes C)$.

The matrix $A^{\otimes n}$ is obtained by induction: $A^{\otimes 1}:=A$, $A^{\otimes n}=A\otimes A^{\otimes (n-1)}$.

For an $N\times M$ matrix $B$, we denote by $\cj B$ the conjugate matrix of $B$, i.e., with entries $\cj B_{ij}:=\cj {(B_{ij})}$.
\end{definition}

\begin{proposition}\label{pr1.2.1}
Let $x$ be in $[0,1]$ and let
$
x=\frac{x_1}{N}+\frac{x_2}{N^2}+\dots,\quad (x_1,x_2,\dots\in\{0,\dots,N-1\})
$ be its base $N$ representation. Let $n=n_1+Nn_2+\dots +N^{p-1}n_p$ be the base $N$ representation of $n$. Then
\begin{equation}
W_n^A(x)=\sqrt{N^p}A^{\otimes p}_{n_1+Nn_2+\dots+N^{p-1}n_{p}, x_1+Nx_2+\dots+N^{p-1}x_{p}}.
\label{eq1.2.1.1}
\end{equation}

\end{proposition}

\begin{proof}
By \eqref{eq1.7}, we just have to prove, by induction, that
\begin{equation}
A^{\otimes p}_{n_1+Nn_2+\dots+N^{p-1}n_{p}, x_1+Nx_2+\dots+N^{p-1}x_{p}}=a_{n_1,x_1}a_{n_2,x_2}\dots a_{n_p,x_p},
\label{eq1.2.1.2}
\end{equation}
for all $p\geq$ and $n_1,\dots,n_p,x_1,\dots,x_p\in\{0,\dots,N-1\}$. This is clearly true for $p=1,2$. Assume it is true for $p$, we have
$$A^{\otimes (p+1)}_{n_1+Nn_2+\dots+N^{p}n_{p+1},x_1+Nx_2+\dots+N^px_{p+1}}=(A\otimes A^{\otimes p})_{n_1+N(n_2+\dots+N^{p-1}n_{p+1}),x_1+N(x_2+\dots+N^{p-1}x_{p+1})}$$$$=a_{n_1,x_1}a_{n_2,x_2}\dots a_{n_{p+1},x_{p+1}}.$$
\end{proof}

\begin{proposition}\label{pr1.4}
Let $n=n_1+Nn_2+\dots N^{p-1}n_p$ be the base $N$ representation of $n$. Suppose the function $f$ in $L^2[0,1]$ is piecewise constant $f_{k_1,\dots,k_p}$ on each interval $$\left[\frac{k_1}{N}+\frac{k_2}{N^2}+\dots+\frac{k_p}{N^p},\frac{k_1}{N}+\frac{k_2}{N^2}+\dots+\frac{k_p}{N^p}+\frac{1}{N^p}\right],$$ $k_1,\dots,k_p\in\{0,\dots,N-1\}$. Then

\begin{equation}
\ip{f}{W_n^A}=\sum_{k_1,\dots,k_p=0}^{N-1}\sqrt{N^p}f_{k_1,\dots,k_p}\cj A^{\otimes p}_{n_1+Nn_2+\dots+N^{p-1}n_p,k_1+Nk_2+\dots+N^{p-1}k_p}.
\label{eq1.4.0}
\end{equation}

\end{proposition}

\begin{proof}
The result follows immediately from Propositions \ref{pr1.2} and \ref{pr1.2.1}.
\end{proof}

\medskip
{\bf Notations and Identifications.}
Fix a scale $N\geq 2$ and some resolution level $p$. We identify a non-negative integer $n\leq N^p-1$ with its base $N$ representation $n=n_1+Nn_2+\dots N^{p-1}n_p$, $n_1,\dots,n_{p-1}\in\{0,\dots, N-1\}$, $n\leftrightarrow n_1n_2\dots n_p$. We index vectors in $\bc^{N^p}$ by numbers between 0 and $N^p-1$ and identify functions which are constant on the intervals $[k/N^p,(k+1)/N^p)$ with their corresponding vector of values.

With these identifications, the formula \eqref{eq1.2.1.1} becomes
\begin{equation}
W_n^A(x)=\sqrt{N^p}A^{\otimes p}_{n,x_1x_2\dots x_p},
\label{eq1.4.1}
\end{equation}
and the formula \eqref{eq1.4.0} becomes a multiplication of the matrix $A^{\otimes p}$ by the vector $f$. 
\begin{equation}
(\ip{f}{W_n^a})_{n=0}^{N^p-1}=\sqrt{N^p}\,\cj A^{\otimes p}f.
\label{eq1.4.2}
\end{equation}
\medskip

\begin{proposition}\label{pr1.3}{\bf [Change of basis between generalized Walsh bases]}
Let $A=(a_{ij})_{i,j=0,1,\dots,N-1}$ and $B=(b_{ij})_{i,j=0,1,\dots,N-1}$ be two $N$ by $N$ unitary matrices with constant first row. Then the change of basis matrix from $\{W_n^A\}_{n\in\bn_0}$ to $\{W_n^B\}_{n\in\bn_0}$ is given by the entries:
\begin{equation}
M_{nm}=\ip{W_n^B}{W_m^A}=(BA^*)_{n_1m_1}(BA^*)_{n_2m_2}\dots (BA^*)_{n_pm_p}=(BA^*)^{\otimes p}_{nm},\quad(m,n\in\bn),
\label{eq1.3.1}
\end{equation}
where $n=n_1+Nn_2+\dots$ and $m=m_1+Nm_2+\dots$ are the base $N$ representations, and $p$ is chosen such that $n_k=m_k=0$ for $k>p$.
\end{proposition}

\begin{proof}
Using Proposition \ref{pr1.2}, \eqref{eq1.8} we have
$$M_{nm}=\sum_{j_1,j_2,\dots,j_p=0}^{N-1}N^pb_{n_1j_1}b_{n_2j_2}\dots b_{n_pj_p}\cj a_{m_1j_1}\cj a_{m_2j_2}\dots a_{m_pj_p}\cdot\frac{1}{N^p}$$
$$=\left(\sum_{j_1=0}^{N-1}b_{n_1j_1}\cj a_{m_1j_1}\right)\left(\sum_{j_2=0}^{N-1}b_{n_2j_2}\cj a_{m_2j_2}\right)\dots\left(\sum_{j_p=0}^{N-1}b_{n_pj_p}\cj a_{m_pj_p}\right)$$$$=(BA^*)_{n_1m_1}(BA^*)_{n_2m_2}\dots (BA^*)_{n_pm_p}$$

\end{proof}

\section{A fast generalized Walsh transform}\label{fast}
Let $I_n$ be the $n\times n$ identity matrix. For a function $f$ which is piecewise constant on intervals of length $1/N^p$, as in Proposition \ref{pr1.4}, to find its coefficients in the generalized basis, according to \eqref{eq1.4.2}, we need a multiplication by the matrix $\cj A^{\otimes p}$. Since this is an $N^p\times N^p$ matrix, each coefficient requires $N^p$ operations (by operation, we mean a multiplication and an addition of complex numbers). Since there are $M:=N^p$ coefficients, we need $(N^p)^2$ operations. However, using ideas from \cite{LeKa86}, we can improve the speed of these computations using a certain factorization, based on the following relations: if $A,A'$ are $n\times n$ matrices and $B,B'$ are $m\times m$ matrices then
$$(A\otimes B)\cdot (A'\otimes B')=(A\cdot A')\otimes (B\cdot B'),$$
and in particular
$$I_n\otimes (B\cdot B')=(I_n\otimes B)\cdot (I_n\otimes B').$$

Therefore, we have
$$\cj A^{\otimes p}=(\cj A\otimes I_{N^{p-1}})\cdot (I_N\otimes \cj A^{\otimes (p-1)})=(I_{N^0}\otimes \cj A\otimes I_{N^{p-1}})\cdot (I_N\otimes  (\cj A\otimes I_{N^{p-2}})\cdot (I_N\otimes\cj A^{\otimes (p-2)}))$$
$$=(I_{N^0}\otimes \cj A\otimes I_{N^{p-1}})\cdot (I_N\otimes\cj A\otimes I_{N^{p-2}})\cdot (I_N\otimes I_N\otimes \cj A^{\otimes (p-2)})$$$$=(I_{N^0}\otimes \cj A\otimes I_{N^{p-1}})\cdot (I_N\otimes\cj A\otimes I_{N^{p-2}})\cdot (I_{N^2}\otimes \cj A^{\otimes (p-2)})=\dots \mbox{ by induction }$$$$=\prod_{k=0}^{p-1}(I_{N^k}\otimes \cj A\otimes I_{N^{p-1-k}}).$$

Note that each term in the last product is a matrix which has at most $N$ non-zero entries in each row/column. Thus for each coefficient, each multiplication will require at most $N$ operations, and since there are $p$ terms in the product, each coefficient will require $Np$ operations. For all the $N^p$ coefficients, one needs $N^p\cdot p\cdot N$ operations, and with $N$ fixed and $p$ variable, and $M=N^p$ this means $O(M\log M)$ operations.

More precisely, we have, for $k=0,\dots,p-1$, and $i_1,j_1=0,1,\dots, N^k-1$, $i_2,j_2=0,1,\dots,N-1$, $i_3,j_3=0,1,\dots,N^{p-1-k}$
$$(I_{N^k}\otimes\cj A\otimes I_{N^{p-1-k}})_{i_1+N^ki_2+N^{k+1}i_3, j_1+N^kj_2+N^{k+1}j_3}=\delta_{i_1,j_1}\cj a_{i_2,j_2}\delta_{i_3,j_3}.$$
Thus, for a vector  $(v^{(k+1)}_i)_{i=0}^{N^p-1}$, if $v^{(k)}=((I_{N_k}\otimes\cj A\otimes I_{N^{p-1-k}})v^{(k+1)}$, then for $i_1=0,1,\dots,N^k-1$, $i_2=0,1,\dots, N-1$, $i_3=0,1,\dots,N^{p-1-k}-1$,
$$v^{(k)}(i_1+N^ki_2+N^{k+1}i_3)=\sum_{j_2=0}^{N-1}\cj a_{i_2,j_2}v^{(k+1)}(i_1+N^kj_2+N^{k+1}i_3).$$
Using these operations $p$ times, starting with a vector $v$, we let $v^{(p)}=v$ and then $\cj A^{\otimes p}v=v^{(0)}$.

Note also that the change of base between two generalized Walsh bases (see Proposition \ref{pr1.3}) can be performed using a fast algorithm, by writing
$$(BA^*)^{\otimes p}=\prod_{k=0}^{p-1}\left(I_{N^k}\otimes (BA^*)\otimes I_{N^{p-1-k}}\right).$$

\section{Uncertainty principles for Walsh transforms}\label{incert}
Let $G $ be the finite abelian group $\bz_m$ with $m$ elements,  and  $f:G\rightarrow \mathbf{C}$ a finite length signal on $G$ with $\hat{f}:G\to\mathbf{C}$ its (discrete) Fourier transform.  The uncertainty principle in this set up relates the support of $f$ and $\hat{f}$ as follows (see \cite{DoSt}):

\begin{equation}\label{unc}
|\text{supp}(f)|\cdot |\text{supp}(\hat{f})|\geq m
\end{equation}

As a simple consequence one obtains
\begin{equation}\label{med}
  |\text{supp}(f)| + |\text{supp}(\hat{f})|\geq 2\sqrt{ m}
\end{equation}

These inequalities are useful in signal recovery in case of sparse signals.  By exploiting a property of Chebotarev on the minors of the Fourier matrix when $m$ is prime in \cite{Taounc} inequality (\ref{med}) is vastly improved:  $|\text{supp}(f)| + |\text{supp}(\hat{f})|\geq  m+1 $.

\par In this section we explore the uncertainty principle (\ref{unc}) for the (classic and generalized) Walsh transforms.
 In the following, we let $A$ be a $N\times N$ unitary matrix with constant row $1/\sqrt{N}$, and $p\geq 1$ an integer. We will denote by $b_{ij}$, $i,j=0,..., N^P-1$  the entries of the matrix $A^{\otimes p}$. In this set-up if $f\in\mathbf{C}^p$ is a signal then $W_A(f)=\sqrt{N^p}A^{\otimes p}f$  encodes its 'frequencies'. Note that $\|A^{\otimes p}f\|=\|f\|$.  Define the generalized Walsh-Fourier transform corresponding to $A$ by
$$Tf(k):=\frac{1}{\sqrt{N^p}}W_A(f) =[ A^{\otimes p}f ](k),\quad(k=0,\dots, N^p-1).$$
The next theorem shows that a certain discrepancy occurs between the classic Walsh and generalized Walsh transforms. While the classic Walsh satisfies (\ref{unc}),  some generalized Walsh transforms may satisfy a weaker version.
\begin{theorem}\label{unw} Let $ N$, $p$ and $A=(a_{ij})_{i,j=0,\dots,N-1}$ as above and $f\in\mathbf{C}^{N^p}$ a non zero vector. If  $\alpha>0$ satisfies  $\max_{i,j}|a_{ij}|\leq (\frac{1}{\sqrt{N}})^{\alpha}$
then :
\begin{enumerate}
	\item The following uncertainty principle holds
	\begin{equation}\label{uncw}
|\text{supp}(f)|\cdot |\text{supp}(Tf)|\geq N^{p\alpha}
\end{equation}
\item If $A$ is the $2\times 2$ matrix giving rise to the classic Walsh transform then:
\begin{equation}\label{uncc}
|\text{supp}(f)|\cdot |\text{supp}(Tf)|\geq N^{p}
\end{equation}
\item
$\alpha=1$ if and only if $\sqrt{N}\cdot A$ is a Hadamard matrix, i.e., a matrix with all entries of absolute value one and orthogonal rows/columns.
\item If $A$ has real entries only and $N>2$ then necessarily $\alpha < 1$ unless $N$ is even and  $A$ is a $1/\sqrt{N}$ multiple of the Hadamard matrix.
\end{enumerate}
\end{theorem}

\begin{proof}
(i) Using the calculations for $A^{\otimes p}$ from section \ref{fast}, we know that its entries satisfy  $|b_{ij}|\leq 1/\sqrt{N^{p\alpha}} $. Now we essentially follow the proof of (\ref{unc}). Using Cauchy-Schwartz we have:
$$ \text{max}_{k}|Tf(k)|\leq \frac{1}{\sqrt{N^{p\alpha}}}\cdot \sum_{j\in\text{supp}(f) }|f(j)|\cdot 1 $$
$$ \leq  \frac{1}{\sqrt{N^{p\alpha}}} ( |\text{supp}(f)|)^{1/2}(\sum_{j}|f(j)|^2)^{1/2}=    \frac{1}{\sqrt{N^{p\alpha}}} ( |\text{supp}(f)|)^{1/2}\|f\|= \frac{1}{\sqrt{N^{p\alpha}}} ( |\text{supp}(f)|)^{1/2}\|A^{\otimes p}f\|$$
$$=\frac{1}{\sqrt{N^{p\alpha}}} ( |\text{supp}(f)|)^{1/2}(\sum_j |Tf(j)|^2)^{1/2}\leq \frac{1}{\sqrt{N^{p\alpha}}} ( |\text{supp}(f)|)^{1/2}( |\text{supp}(Tf)|)^{1/2} \text{max}_{k}|Tf(k)|$$
Retaining the first and last terms and simplifying ($f$ is non zero)  we get (\ref{uncw}).

(ii) follows from (i)  because the entries of the (unique) $2\times 2$ matrix giving rise to classic Walsh system satisfy  $|a_{ij}|=1/\sqrt{2}$. Hence $\alpha=1$ is applicable in (i).

(iii) If $A$ is a $N\times N$ unitary matrix having constant $1/\sqrt{N}$ first row then $max|a_{ij}|\geq 1/\sqrt{N}$. However each row must have norm $1$. This means that, in case $\alpha=1$, all  $|a_{ij}|$ must be equal to $1/\sqrt{N}$ thus $\sqrt{N}\cdot A$ is Hadamard.

(iv) follows from (iii) and the fact that the entries in this case must be $\pm{1}$.
\end{proof}

Using the inequality between the arithmetic and geometric means, we obtain:
\begin{corollary}
$$|\suport(f)|+ |\suport(Tf)|\geq 2 N^{p\alpha/2}$$
\end{corollary}

%

\begin{example} If $f$ is a  Dirac signal then equality is attained in $(\ref{uncc})$. This follows easily from the form of the matrix $A^{\otimes p}$ with $A$ as in (\ref{2by2}). The Dirac signal also produces equality in  (\ref{uncw}) when $\sqrt{N}\cdot A$ is the $N\times N$ Fourier matrix. The inequality (\ref{uncw}) is not always optimal. For example with $A$ as in (\ref{3by3}) we have $\text{max}|a_{ij}|=\sqrt{6}/3$ and $\alpha\approx 0.369$. With $p=3$ and $f_i=0$ for all $i$ but $f_{14}=1$ inequality (\ref{uncw}) simply says $8>3.375$.
\end{example}

Similarly to the context of the discrete Fourier transform \cite{DoSt}, we exploit uncertainty to show that sparse signals can be uniquely recovered when a few transform coefficients are lost.
\begin{theorem}\label{recov}  With the matrix $A$ and the constant $\alpha$ as above, consider a signal $f\in\mathbf{C}^{N^p}$  of which the following are known:

\begin{enumerate}
	\item  The number of non zero components of $f$, $N_f=|\suport(f)|$.
	\item A subset $B\subseteq \suport(Tf)$ of observed 'frequencies'  with which form the signal
$$\tilde{f}(k):=  \left\{ \begin{array}{cc}
 Tf(k),    &\text{ if  } k\in B\\
0  ,&\text{ otherwise}
\end{array}\right.
$$
\item The  number of unobserved 'frequencies'  $N_w=|B^c|$ satisfies
 \begin{equation}\label{miss}
 2N_fN_w<N^{p\alpha}
 \end{equation}

\end{enumerate}

Then $f$ can be uniquely reconstructed from data (i), (ii) and (iii).
\end{theorem}
\begin{proof} We prove uniqueness: if $g\in \mathbf{C}^{N^p}$ is such that $N_f=N_g$ and $\tilde{f}=\tilde{g}$, then the signal
$h:=f-g$ satisfies
$$|\text{supp}(h)|\leq N_f+N_g=2N_f$$
$$| \text{supp}(Th)|\leq N_w\text{ because } Th(k)=0\text{ for } k\in B$$
It follows that $ |\text{supp}(h)|\cdot |\text{supp}(Th)| <N^{p\alpha}$. This contradicts theorem (\ref{unw})  unless $h=0$.\\
\end{proof}
\begin{remark}
As in \cite{DoSt}, reconstruction of $f$ is possible by solving the min-problem  $$\min \{ \| T^{-1}\tilde{f}-v\|: \text{supp}(v)\leq N_f  \}.$$ The minimum must be attained and from uniqueness this happens when $v=f$.

In Theorem \ref{unw} we must have $\alpha\leq 1$. In terms of recovery by using uncertainty the best value is $\alpha=1$, otherwise inequality (\ref{miss}) will severely restrict the number $N_w$ of frequencies that can be missed or lost in transmission.
It would be interesting to find efficient recovery algorithms similar to those in \cite{DoSt}. As mentioned in that paper, solving the min problem by brute force is inefficient. \\
\end{remark}

One can generalize Theorem \ref{unw} by considering unitary matrices regardless of constant first row (which was necessary in order to obtain a generalized Walsh basis, see \cite{DPS14}). For a given unitary matrix we introduce its uncertainty constant. As we will see below, this constant satisfies some stability properties with respect to tensor products.

\begin{theorem}\label{thu1}
Let $A$ be a $n\times n$ unitary matrix. Let $M:=\max_{i,j}|A(ij)|$. Then
\begin{equation}
|\textup{supp}(f)|\cdot|\textup{supp}(Af)|\geq \frac1{M^2},\quad (f\in \bc^n\setminus\{0\}).
\label{equ1.1}
\end{equation}
\end{theorem}

\begin{proof}
Let $f\in\bc\setminus \{0\}$. We have, using the Cauchy-Schwarz inequality:
$$\max_k|Af(k)|=\max_k\left|\sum_j A(kj)f(j)\right|\leq M\sum_{j\in\textup{supp}(f)}|f(j)|\cdot 1\leq M\left(|\textup{supp(f)}|\right)^{\frac12}\|f\|$$
$$= M\left(|\textup{supp}(f)|\right)^{\frac12}\|Af\|=M\left(|\textup{supp}(f)|\right)^{\frac12}\left(\sum_j|Af(j)|^2\right)^{\frac12}$$
$$=M\left(|\textup{supp}(f)|\right)^{\frac12}\left(|\textup{supp}(Af)|\right)^{\frac12}\max_k |Af(k)|.$$
This implies \eqref{equ1.1}.
\end{proof}

\begin{definition}\label{defu2}
Let $A$ be an $n\times n$ unitary matrix. We define the {\it uncertainty constant} of $A$ to be
$$\mu(A):=\min_{f\in\bc^n\setminus\{0\}}|\textup{supp}(f)|\cdot |\textup{supp}(Af)|.$$
\end{definition}

\begin{corollary}\label{coru3}
For a unitary matrix $A$, on has
\begin{equation}
\mu(A)\geq \frac{1}{(\max_{i,j}|A(ij)|)^2}.
\label{equ3.1}
\end{equation}

If $\sqrt{N}A$ is an $N\times N$ Hadamard matrix, then
\begin{equation}
\mu(A)=N.
\label{equ3.2}
\end{equation}
\end{corollary}

\begin{proof}
\eqref{equ3.1} follows from Theorem \ref{thu1}. If $\sqrt{N}A$ is Hadamard, then $\max_{i,j}|A(ij)|=\frac1{\sqrt N}$ so $\mu(A)\geq N$. On the other hand $|\suport (A\delta_1)|=N$, so $\mu(A)\leq 1\cdot N$, and \eqref{equ3.2} follows.
\end{proof}

With the relation \eqref{equ3.2}, we obtain an interesting corollary about Hadamard matrices.

\begin{corollary}\label{coru3.2}
Let $\sqrt NA$ be an $N\times N$ Hadamard matrix, and assume $N_1\cdot N_2<N$, with $N_1,N_2\in\bn$, and let $n=N-N_2$. Then a matrix obtained from $A$ by taking $n$ rows and $N_1$ columns has rank $N_1$.
\end{corollary}

\begin{proof}
Let $B$ be the matrix obtained from $\sqrt NA$ by taking the rows $i_1,\dots,i_n$ and the columns $j_1,\dots,j_{N_1}$. Assume that the rank of $B$ is strictly less then $N_1$. That means the columns of $B$ are linearly dependent so there is a non-zero vector $v=\sum_{k=1}^{N_1}v_k\delta_{j_k}$ such that $Bv=0$. Take then $f$ in $\bc^N$ by completing the vector $v$ with zeros. Then $Av$ is zero on the components $i_1,\dots, i_n$. Thus $|\suport(f)|\leq N_1$ and $|\suport (Af)|\leq N-n=N_2$. Then
$$|\suport(f)|\cdot|\suport(Af)|\leq N_1\cdot N_2< N=\mu(A).$$
This contradicts \eqref{equ3.2}. So $B$ has rank $N_1$.
\end{proof}

\begin{proposition}\label{pru4}
If $A_1$ and $A_2$ are unitary matrices, then
\begin{equation}
\mu(A_1\otimes A_2)\leq \mu(A_1)\cdot\mu(A_2).
\label{equ4.1}
\end{equation}
In particular, for a unitary matrix $A$ and $p\in\bn$,
\begin{equation}
\mu(A^{\otimes p})\leq \mu(A)^p.
\label{equ4.2}
\end{equation}
\end{proposition}

\begin{proof}
Let $A_1$ be $n_1\times n_1$ and $A_2$ be $n_2\times n_2$ unitary matrices. Let $f_1\in\bc^{n_1}\setminus \{0\}$ and $f_2\in \bc^{n_2}\setminus\{0\}$ be such that
$$\mu(A_1)=|\suport(f_1)|\cdot|\suport(A_1f_1)|,\quad \mu(A_2)=|\suport(f_2)|\cdot|\suport(A_2f_2)|.$$
We have, for $0\leq i_1\leq n_1-1$, $0\leq i_2\leq n_2-1$,
$$(f_1\otimes f_2)(i_1+n_1i_2)=f_1(i_1)f_2(i_2),$$
so $(f_1\otimes f_2)(i_1+n_1i_2)\neq 0$ if and only if $f_1(i_1)\neq 0$ and $f_2(i_2)\neq 0$. Therefore
$$|\suport (f_1\otimes f_2)|=|\suport(f_1)|\cdot|\suport(f_2)|.$$
Then
$$\mu(A_1\otimes A_2)\leq |\suport(f_1\otimes f_2)|\cdot |\suport( (A_1\otimes A_2)(f_1\otimes f_2))|$$
$$=|\suport(f_1)|\cdot |\suport(A_1f_1)|\cdot |\suport(f_2)|\cdot |\suport(A_2f_2)|=\mu(A_1)\cdot\mu(A_2).$$

The inequality \eqref{equ4.2} follows from \eqref{equ4.1} by induction.
\end{proof}

\begin{corollary}\label{coru5}
Let $A_1$ and $A_2$ be unitary matrices and $M_1:=\max_{i,j}|A_1(ij)|$, $M_2:=\max_{i,j}|A_2(ij)|$. Then
\begin{equation}
\frac{1}{M_1^2M_2^2}\leq \mu(A_1\otimes A_2)\leq \mu(A_1)\cdot \mu(A_2).
\label{equ5.1}
\end{equation}

In particular, if $A$ is a unitary matrix and $M:=\max_{i,j}|A(ij)|$, then, for $p\geq 1$,
\begin{equation}
\frac{1}{M^p}\leq \mu(A^{\otimes p})\leq \mu(A)^p.
\label{equ5.2}
\end{equation}
If $A$ is an $N\times N$ Hadamard matrix then
\begin{equation}
\mu(A^{\otimes p})=N^p.
\label{equ5.3}
\end{equation}
\end{corollary}

\begin{proof}
Note that
$$\max_{i,j}|(A_1\otimes A_2)(ij)|=\max_{i,j}|A_1(ij)|\cdot \max_{i,j}|A_2(ij)|.$$
Then the result follows from Theorem \ref{thu1} and Proposition \ref{pru4}. The inequality \eqref{equ5.2} follows from \eqref{equ5.1} by induction. If $A$ is a Hadamard matrix then so is $A^{\otimes p}$, and \eqref{equ5.3} follows from \eqref{equ3.2}.
\end{proof}

\begin{example}\label{exu7}
Let us investigate the uncertainty constant in dimension 3. Consider a $3\times 3$ matrix $A$ with constant first row $\frac1{\sqrt{3}}$.

 $\mu(A)=1$. If $\mu(A)=1$ this means that there exists $f\in \bc^3$ with $|\suport (f)|=|\suport(Af)|=1$. Then $f=\delta_a$ and $Af=\delta_b$ for some $a,b\in\{1,2,3\}$, which means that one of the columns of $A$ is a canonical vector $\delta_b$. Obviously, this cannot happen if the first row is constant, because, then, one of the columns is $(1,0,0)^T$ (say the first one), and reading the rows in the other columns, we get that the vectors $(1,1)$, $(A(22),A(23))$ and $(A(32),A(33))$ are orthogonal, which is impossible.

$\mu(A)=2$. If $\mu(A)=2$ then (case 1) there exists $f\in\bc^3$ with $|\suport(f)|=1$ and $|\suport(Af)|=2$ or (case 2) there exists $f\in\bc^3$ with $|\suport(f)|=2$ and $|\suport(Af)=1$.

In case 1, $f=\delta_i$ and, by a permutation, we can assume $f=\delta_1$ and $Af$ has the zero, so the first column of $A$ has a zero, and, by a permutation, we can assume that it is on the second row. Therefore the matrix has the form
$$\begin{pmatrix}
	\frac{1}{\sqrt{3}}&\frac{1}{\sqrt{3}}&\frac{1}{\sqrt{3}}\\
	0&a&b\\
	c&d&e
\end{pmatrix}.
$$

Since the first two rows are orthogonal we get that $a+b=0$ so $b=-a$. Multiplying the row by a scalar of absolute value 1, we can assume $a$ is real and positive. Since the row has norm 1, we get that $2a^2=1$ so $a=\frac{1}{\sqrt{2}}$, $b=-\frac{1}{\sqrt{2}}$. Since the last two rows are orthogonal, we get $d-e=0$ so $e=d$. As before, we can assume $d$ is real and positive. Since row 1 and row 3 are orthogonal, we get $c+2d=0$ so $c=-d$. Since the norm of the row 3 is 1, we obtain $6d^2=1$ so $d=\frac{1}{\sqrt{6}}$. Thus the matrix is
\begin{equation}\label{equ7.1}
A=\begin{pmatrix}
	\frac{1}{\sqrt{3}}&\frac{1}{\sqrt{3}}&\frac{1}{\sqrt{3}}\\
	0&\frac{1}{\sqrt{2}}&-\frac{1}{\sqrt{2}}\\
	-\frac{2}{\sqrt{6}}&\frac{1}{\sqrt{6}}&\frac{1}{\sqrt{6}}
\end{pmatrix}.
\end{equation}

In case 2, we have $Af=\delta_a$ and $|\suport(f)|=2$. Then $|\suport(A^*\delta_a)|=2$ so a row in $A$ has a zero, which means a column in $A$ has a zero, and we are back to case 1.

In general, of course if we take $f=\delta_1$ then $|\suport(f)|\cdot|\suport(Af)|=3$ so $\mu(A)\leq 3$. Therefore, for all $3\times 3$ matrices that cannot be obtained from the matrix in \eqref{equ7.1} by permutation of columns, permutations of the last two rows, or multiplications of the last two rows by unimodular constants, we have that the uncertainty constant $\mu(A)=3$.

\end{example}

\section{Generalized Walsh vs. classic Walsh and DCT}\label{s6}
In this section we take a look at how the generalized Walsh transforms compare statistically to the Walsh and DCT (discrete cosine) transforms. Recall that the classic Walsh transform corresponds to the (unique) $2\times 2$ unitary matrix $A$ with constant $1/\sqrt{2}$ first row:
\begin{equation}\label{2by2}
A:=\left( \begin {array}{cc} \frac{1}{\sqrt{2}}&\frac{1}{\sqrt{2}}
\\ \noalign{\medskip} \frac{1}{\sqrt{2}}& -\frac{1}{\sqrt{2}}\end {array}
 \right)
\end{equation}
We will do this analysis on 1-dimensional signals $X\in\mathbf{R}^n$.
We will implement the following compression scheme  under  a fixed orthogonal transform $T:\mathbf{R}^n\to\mathbf{R}^n$, i.e. $T^{*}T=I$. This scheme is inspired by the so-called "variance criterion" (see e.g. \cite{ AhRa})  however we do not consider classes of  signals and covariance matrices built from expected values. Instead we use the straightforward  'covariance' matrices below. In the following $v^{t}$ denotes a column vector (the transpose of row vector $v$).
\begin{itemize}
\item Fix an integer $M<n$ ;
\item For a column vector $X=(x_1,x_2,\dots,x_n)^t$ consider $Y=TX=(y_1,y_2,\dots,y_n)^t$  in $\mathbf{R}^n$ ;
\item Calculate the matrix $\text{cov}(Y):=YY^t$ ; \\
From the diagonal pick the highest $M$ 'variances' $\sigma_{i_1,i_1}$, $\sigma_{i_2,i_2}$,...,$\sigma_{i_M,i_M}$ ; \\
Form the compressed signal $\tilde{Y}:=(0,..,0,y_{_{i_1}},0,..,y_{_{i_2}},\dots,0, y_{_{i_M}},0,..0)$ ;
\item $\tilde{X}:=T^{-1}(\tilde{Y})=T^*(\tilde{Y})$ is an approximation of $X$.\\
\item Graph the (normalized) variance distribution
$$k\to \frac{\sigma_{i_k, i_k}}{\text{Trace}(YY^t)}$$
to visualize how efficient is $T$ when $k$ components are kept ($n-k$ removed) corresponding to the highest $k$ variances.
\end{itemize}
The reason this scheme should work is  based on the fact that the minimum error is achieved  when the transform matrix is made of the eigenvectors of $\text{cov}(X)$ i.e. $\text{cov}(Y)=T\text{cov}(X)T^*$ is a diagonal matrix with entries the eigenvalues of $\text{cov}(X)$. We prove this result below by adapting ideas from \cite{AhRa} pp.201-202. We do not have to deal with the expected value operator in the case of a fixed signal. Still,  keeping the highest variances achieves compression. The error function we wish to minimize will  be expressed in simple scalar products.

\begin{proposition} Let $X\in\mathbf{R}^n$ be a  fixed vector and $M< n$ a positive integer. For any orthogonal transform $T:\mathbf{R}^n\to\mathbf{R}^n$ consider the vectors $Y:=TX$ and $\tilde{Y}:=(y_1,y_2,...,y_M, 0,...0)^t$,  and define the error   $E_T:=\|X-T^*\tilde{Y}\|^2$. Then:
\begin{enumerate}
	\item $\text{cov}(Y)=T\text{cov}(X)T^*$;
	\item $\min\{ E_T \text{ }| \text{ } T^*T=I \}$ is attained when $T^*$ is the transform whose columns   $v_1^t, v_2^t, ..., v_n^t$  are the eigenvectors of $\text{cov}(X)$. Hence $\text{cov}(Y)$ is a diagonal matrix.
\end{enumerate}
\end{proposition}

\begin{proof} (i) The formula follows easily from the definition $\text{cov}(Y)=TX(TX)^t$ and $T^t=T^* $ because $T$ has real valued entries. \\

(ii) For an arbitrary $T$, let $v_i^t$ be the column vectors of $T^*$. If $Y=TX=(y_1,\dots,y_n)^t$, then $X=T^*Y=\sum_{i=1}^ny_iv_i^t$. Because $T$ is orthogonal

$$E_T=\|X-T^*\tilde Y\|^2=\|T^*Y-T^*\tilde Y\|^2=\|Y-\tilde Y\|^2=\sum_{i=M+1}^ny_i^2=\sum_{i=M+1}^n\ip{v_i^t}{X}^2.$$
The minimum is subject to the constraints $v_iv_i^t=1$ (we discard the orthogonality conditions $v_iv_j^t=0$ for $i\neq j$, because, as we will see, the minimum, under these conditions, will be realized for eigenvectors of cov(X), which can be chosen to be orthogonal).

With $\lambda_i$ being the Lagrange multipliers we look for critical points of the function
$$\tilde E(v)=\sum_{i=M+1}^n\ip{v_i^t}{X}^2-\sum_{i=M+1}^n\lambda_i(v_iv_i^t-1),$$
in variables $(v_i)_{i=1,\dots,n}$.

Note that, for row vectors $v$, if $f(v)=\ip{v^t}{X}^2$, then $\nabla_v(f)=2\ip{v^t}{X}X=2(XX^t)v^t=2\textup{cov}(X)v^t$. Also $\nabla_v (vv^t)=2v^t$. Then $\nabla_{v_i}(\tilde E)=0$ implies $\textup{cov}(X)v_i^t=XX^tv_i^t=\lambda_i v_i^t$.

In conclusion, the minimum is obtained when the column vectors of $T^*$ are eigenvectors for $\textup{cov}(X)$, and therefore $\textup{cov}(Y)=T\textup{cov}(X)T^*$ is a diagonal matrix in the canonical basis.
\end{proof}

By deleting a component $y_j$ from the transformed signal we mean setting $y_j=0$. If $T^*$ diagonalizes $\text{cov}(X)$  the removal of the $j^{\text{th}}$ component will result in a mean square error increase by the  corresponding eigenvalue $\lambda_j$. To achieve compression it makes sense to discard the lowest $n-M$ eigenvalues. This optimum transform (Karhunen-Loeve) is not cheap to implement, however. The implementation of Walsh and other transforms is more suitable (we also  analyze  DCT by the same method). For these transforms the covariance matrix has non zero off-diagonal terms, nevertheless by discarding the lowest values from its diagonal  still produces compression, as seen in the following examples.
\begin{example}
We illustrate the above principle for $T_1=A^{\otimes 6}$ and $T_2=B^{\otimes 6}$ with $3\times 3$ matrices A and B below (when entries contain decimals (for the matrix $B$), the matrices involved are 'almost' unitary because our Maple code approximates the solutions of the orthogonality equations). The signal $X$ is a vector of length $3^6$ with values in $[0,1]$ (the vector is a column in a black and white  image). We set to zero  $90\%$  of  $T_iX$ components and  keep $70$ components out of $729$ that correspond to the highest variances.  In Figure \ref{warp33}, the signal $X$ is shown with its approximations where the  mean square  error is $0.84$ under transform $T_1$ and $0.95$ under $T_2$. We also graph the variance of both transforms with respect to number of components ($x-$axis). The area under each curve for a given number of components indicates the energy contained in those components : e.g. with respect to component interval $[1,20]$ transform $T_1$ performs better than $T_2$ .  
\begin{equation}\label{3by3}
A:=\left( \begin {array}{ccc} \frac{1}{\sqrt{3}}&\frac{1}{\sqrt{3}}&\frac{1}{\sqrt{3}}\\
\noalign{\medskip} 0 & \frac{1}{\sqrt{2}}&  -\frac{1}{\sqrt{2}}\\
 \noalign{\medskip}   -\frac{2}{\sqrt{6}} & \frac{1}{\sqrt{6}} & \frac{1}{\sqrt{6}}
\end {array} \right),\quad\quad
B:=\left( \begin {array}{ccc}  \frac{1}{\sqrt{3}}&\frac{1}{\sqrt{3}}&\frac{1}{\sqrt{3}} \\
\noalign{\medskip}- 0.2&- 0.58& 0.78
\\ \noalign{\medskip}- 0.79& 0.57& 0.22
\end {array} \right)
\end{equation}

\begin{figure}

    \includegraphics[width=48mm]{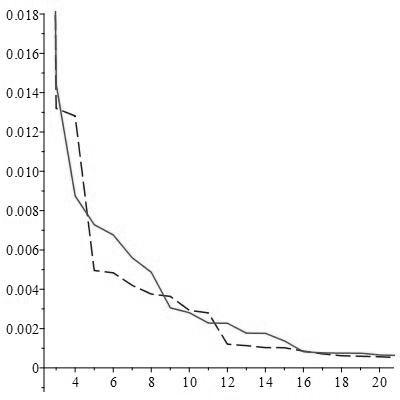}
		\includegraphics[width=48mm]{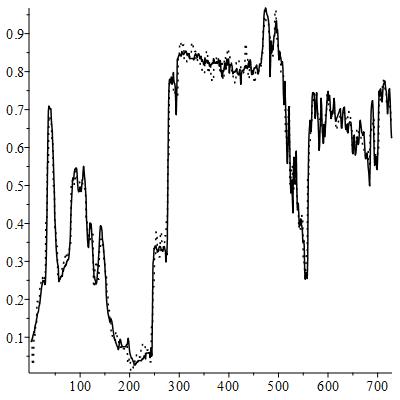}
    \includegraphics[width=48mm]{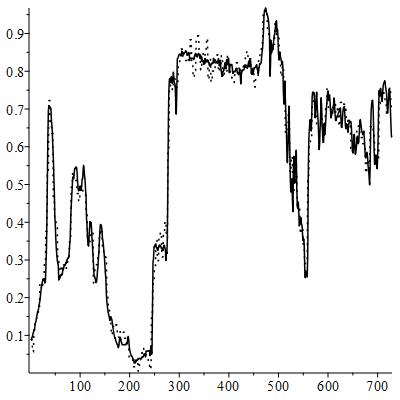}
\caption{Variance distribution ( left) for $T_1$-solid curve packs more energy than $T_2$ -dash. Approximation of a signal with $T_1$ (middle) and $T_2$(right)}
\label{warp33}
\end{figure}
\end{example}

\begin{example}
We compare now  the DCT, the classic  Walsh and a generalized Walsh transform based on the 4 by 4 matrix
\begin{equation}\label{4by4}
\left( \begin {array}{cccc} \frac{1}{2} & \frac{1}{2} & \frac{1}{2} & \frac{1}{2}\\
\noalign{\medskip}  \frac{1}{\sqrt{2}} & -\frac{1}{\sqrt{2}}   & 0  & 0\\
 \noalign{\medskip}  0  & 0 & \frac{1}{\sqrt{2}}&-\frac{1}{\sqrt{2}}\\
 \noalign{\medskip}   \frac{1}{2} & \frac{1}{2}  & -\frac{1}{2 }&-\frac{1}{2}
\end {array} \right)
\end{equation}
In this example we choose a  $2^8$ length vector with high variation defined as follows
$$
X(i):=\left\{ \begin{array}{cc}
\frac{i}{3i+1},&\text{ if }i\text{}|\text{}9\\
\frac{i}{i+1},&\text{ otherwise}
\end{array}\right.
$$
In Figure \ref{wardct} the variance distribution of the transforms (normalized variances in decreasing order of magnitude with respect to transform components) is depicted in order to check that the best error is obtained for the transform whose variance is higher. We keep $45\%$ of the $TX$ components and replace the rest with zeros. The best approximation (see Figure \ref{warpd42}) is given by DCT with error $0.01$, followed by classic Walsh with error $0.08 $ and the generalized Walsh with error $0.33$.

\begin{figure}
\centering
    \includegraphics[width=48mm]{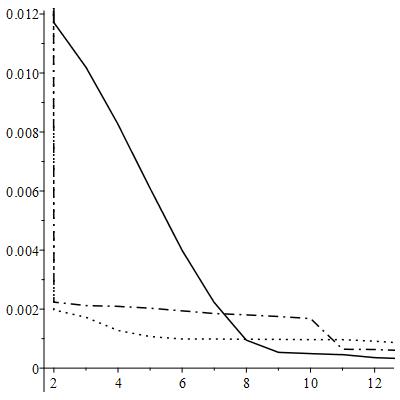}
\caption{Variance distribution: DCT(solid), Walsh(dash) and generalized $4\times 4$ Walsh(dot) }
\label{wardct}
\end{figure}

\begin{figure}
 \includegraphics[width=45mm]{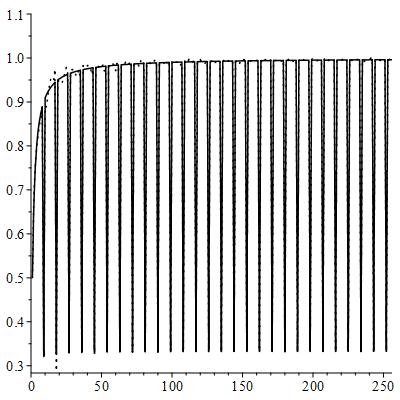}
   \includegraphics[width=45mm]{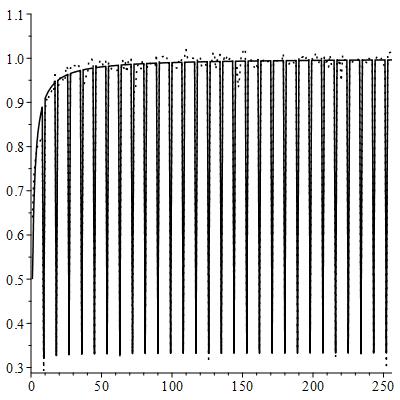}
 \includegraphics[width=45mm]{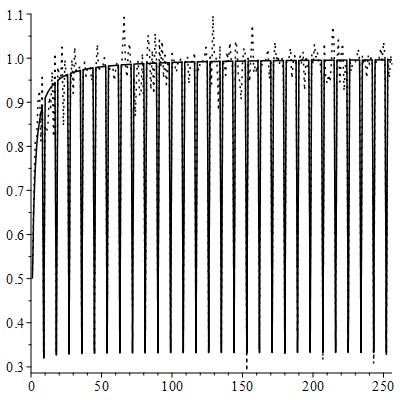}
\label{warpd42}
\caption{Approximation with DCT(left), classic Walsh(middle) and generalized 4 by 4 Walsh (right) with $45\%$ of components }
\end{figure}

\begin{figure}
 \includegraphics[width=48mm]{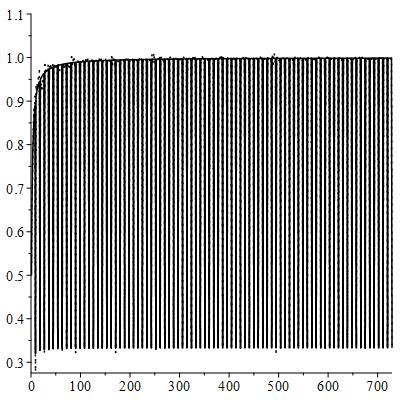}
   \includegraphics[width=48mm]{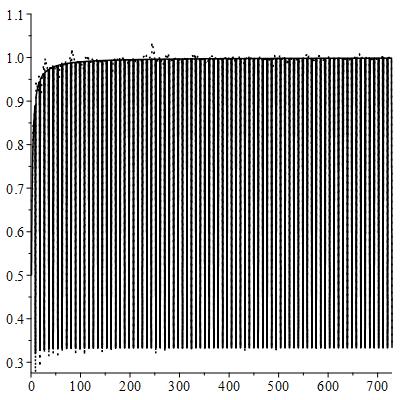}
\label{hvab}
\caption{Approximation of a high variation signal with two generalized 3 by 3 Walsh transforms. $90\%$ components are set to $0$. Recovery errors $0.02$ and $0.08$ }
\end{figure}

It seems that high variation signals are treated better by DCT than (generalized) Walsh. However with the same signal now extended to $729$ components  $X_i, i=1...3^6$ , the generalized  Walsh transforms associated with the 3 by 3 matrices $A$ and $B$ above produce  compression strikingly different from DCT,  see Figure \ref{hvab}.  The errors in recovering $X$ after removing again $90\%$ from $T_iX$ with $T_1=A^{\otimes 6}$ and $T_2=B^{\otimes 6}$ are $0.02$ and $0.08$ respectively. Thus $T_1$ is more efficient than $T_2$ however both are stronger than DCT, the classic Walsh and the generalized Walsh (associated to the 4 by 4 matrix above ) when $90\%$ of $X_i, i=1..2^8$ components are zeroed out. In this case the latter transforms produce errors  $ 0.38$, $5.25$ and $ 6.43$ respectively.
\end{example}

The phenomenon in the example above suggests the following question: given the dimension of the matrix $N$ and a prescribed signal $X$ in $\mathbb R^{N^p}$, which $N\times N$ unitary matrix $A$ will produce a compression-efficient generalized Walsh transform?

\begin{acknowledgements}
This work was partially supported by a grant from the Simons Foundation (\#228539 to Dorin Dutkay).

\end{acknowledgements}

\bibliographystyle{alpha}	
\bibliography{eframes}
\end{document}